\documentclass[11pt,a4paper]{amsart}%
\usepackage[utf8]{inputenc}
\usepackage[T2A]{fontenc}
\usepackage[russian,english]{babel}
\usepackage{amssymb}
\usepackage{amsthm,amsmath}
\usepackage{graphicx}
\usepackage{amsfonts}
\usepackage{amssymb,latexsym}
\usepackage{url}
\usepackage{amsmath}
\usepackage{cite}
\usepackage{ifthen}
\usepackage{tikz}
\usetikzlibrary{calc}
\usepackage{intervalsTop}
\usepackage{enumitem}

\newtheorem{theorem}{Theorem}[section]

\newtheorem{corollary}[theorem]{Corollary}

\newtheorem{lemma}[theorem]{Lemma}
\theoremstyle{definition}
\newtheorem{definition}[theorem]{Definition}

\theoremstyle{remark}
\newtheorem{remark}[theorem]{Remark}

\newcommand{\DeclareMathOperatorWithArity}[2]{%
  \makeatletter%
  \expandafter\DeclareMathOperator\csname #1@@\endcsname{#2}%
  \expandafter\newcommand\csname #1\endcsname[1][]{\ifthenelse{\equal{##1}{}}{\csname #1@@ \endcsname}{\csname #1@@ \endcsname^{(##1)}}}%
  \makeatother
}
\newcommand{\DeclareMathOperatorWithArityAndArg}[3]{%
  \makeatletter%
  \expandafter\DeclareMathOperator\csname #1@@\endcsname{#2}%
  \expandafter\newcommand\csname #1\endcsname[1][]{\ifthenelse{\equal{##1}{}}{\csname #1@@ \endcsname(#3)}{\csname #1@@ \endcsname^{(##1)}(#3)}}%
  \makeatother
}

\newcommand{\2}{{\mathbf{2}}}
\DeclareMathOperatorWithArityAndArg{PA}{Par}{A}
\DeclareMathOperatorWithArityAndArg{ParBool}{Par}{\2}
\DeclareMathOperatorWithArityAndArg{OA}{Op}{A}
\DeclareMathOperatorWithArityAndArg{OpBool}{Op}{\2}
\DeclareMathOperatorWithArityAndArg{JA}{Proj}{A}
\DeclareMathOperatorWithArityAndArg{RelsA}{Rel}{A}
\DeclareMathOperatorWithArityAndArg{RelsBool}{Rel}{\2}
\DeclareMathOperatorWithArityAndArg{cPairsA}{Pair}{A}
\DeclareMathOperatorWithArityAndArg{cPairsBool}{Pair}{\2}
\DeclareMathOperatorWithArity{Pol}{Pol}
\DeclareMathOperatorWithArity{Inv}{Inv}
\DeclareMathOperatorWithArity{pPol}{pPol}
\DeclareMathOperatorWithArity{cPol}{cPol}
\DeclareMathOperatorWithArity{cInv}{cInv}

\DeclareMathOperator{\clone}{Clone}

\newcommand{\DeclareNewVector}[2]{\expandafter\newcommand\csname #1\endcsname{{\mathbf{#2}}}}
\DeclareNewVector{xx}{x}
\DeclareNewVector{yy}{y}
\DeclareNewVector{zz}{z}
\DeclareNewVector{ii}{i}
\DeclareNewVector{0}{0}
\DeclareNewVector{1}{1}

\newcommand{\IN}{{\mathbb{N}}}
\DeclareMathOperator{\dom}{dom}
\DeclareMathOperator{\pr}{pr}
\DeclareMathOperator{\StrN}{Str}
\newcommand{\Str}[1]{\StrN(#1)}
\newcommand{\intervalD}[1]{\ensuremath{\mathcal I(#1)}}
\newcommand{\intervalStr}[1]{\ensuremath{\mathcal I_{\StrN}(#1)}}
\newcommand{\intervalStrUp}[1]{\ensuremath{\mathcal I_{\StrN}^{\subseteq}(#1)}}
\DeclareMathOperator{\Arity}{Arity}

\newcommand{\atPos}[2]{\underset{\substack{\uparrow \\ \makebox[0pt]{\ensuremath{#1}}\vphantom{\ensuremath{#1}}}}{#2}}

\newcommand{\Ptwoone}{$\Omega_1$}
\newcommand{\PostsLattice}[1]{
  \begin{tikzpicture}[scale=#1, transform shape]
    \tikzstyle{every node} = [circle, fill=black,scale=0.5]
    \tikzstyle{every label} = [scale=2,draw=none, fill=none, label distance=-4]
    \node (P2) [label=above:$\OpBool$] at (7.0   ,10.0) {};
    \node (T0) [label=above left:$T_0$]  at (7.0-1.5,10.0-.5) {};
    \node (T1) [label=above right:$T_1$] at (7.0+1.0,10.0-.5) {};
    \node (T)  at (7.0-.5,10.0-1.0) {};
    \node (M) [label=above left:$M$] at (6.76   ,9.38) {};
    \node (T0M) at (7.23-1.5, 9.51-.5) {};
    \node (T1M) at (6.76+1.0, 9.38-.5) {};
    \node (TM)  at (7.23-.5, 9.51-1.0) {};
    \node (L) [label=above right:$L$] at (7.0   , 5.2) {};
    \node (T0L) at (7.0-1.5, 5.2-.5) {};
    \node (T1L) at (7.0+1.0, 5.2-.5) {};
    \node (TL)  at (7.0-.5, 5.2-1.0) {};
    \node (AV)  at (7.0   , 3.0) {};
    \node (T0AV) at (7.0-1.5, 3.0-.5) {};
    \node (T1AV) at (7.0+1.0, 3.0-.5) {};
    \node (TAV)  at (7.0-.5, 3.0-1.0) {}; 
    
    \node (P21) [label=right:\Ptwoone] at (7.0,3.8) {}; 
    \node (S) [label=right:$S$] at (6.85,7.0) {};
    \node (SL) at (6.75,4.7) {};
    \node (SP21) at (6.65,2.5) {};
    \node (ST) at (6.5,6.5) {};
    \node (SM) at (6.15,6.0) {};

    \foreach \from/\to in {
          P2/T0, P2/T1, T0/T, T1/T,
          M/T0M, M/T1M, T0M/TM, T1M/TM,
          P2/M, T0/T0M, T1/T1M, T/TM,
          L/T0L, L/T1L, T0L/TL, T1L/TL,
          P2/L, T0/T0L, T1/T1L, T/ST, ST/TL,
          AV/T0AV, AV/T1AV, T0AV/TAV, T1AV/TAV,
          L/P21, P21/AV, T0L/T0AV, T1L/T1AV, TL/TAV,
          P2/S, S/ST, S/SL, L/SL, SL/TL, SL/SP21, P21/SP21, SP21/TAV,
          ST/SM, SM/TAV}
    \draw [-] (\from) -- (\to);

    \node (T02) [label=left:$T_{0,2}$] at (.5, 8.5) {};
    \node (T02T) at (.5+1.0, 8.5-.7) {};
    \node (T02M) at (.5+2.0, 8.5-.3) {};
    \node (T02TM) at (.5+3.0, 8.5-1.0) {};
    \node (T03) [label=left:$T_{0,3}$] at (.5, 7.5) {};
    \node (T03T) at (.5+1.0, 7.5-.7) {};
    \node (T03M) at (.5+2.0, 7.5-.3) {};
    \node (T03TM) at (.5+3.0, 7.5-1.0) {};
    \node (T0H) [draw=none, fill=none, scale=0.1] at (.5, 7.0) {};
    \node (T0HT) [draw=none, fill=none, scale=0.1] at (.5+1.0, 7.0-.7) {};
    \node (T0HM) [draw=none, fill=none, scale=0.1] at (.5+2.0, 7.0-.3) {};
    \node (T0HTM) [draw=none, fill=none, scale=0.1] at (.5+3.0, 7.0-1.0) {};
    \node (T0e) [label=left:$T_{0,\infty}$] at (.5, 5.5) {};
    \node (T0eT) at (.5+1.0, 5.5-.7) {};
    \node (T0eM) at (.5+2.0, 5.5-.3) {};
    \node (T0eTM) at (.5+3.0, 5.5-1.0) {};
    
    \node (A) [label=below:$\Lambda$] at (.5+3.5,3.7) {};
    \node (AT1) at (.5+4.5,3.2) {};
    \node (AT0) at (.5+2.0,3.2) {};
    \node (AT) at (.5+3.0,2.7) {}; 

    \node (T12) [label=right:$T_{1,2}$] at (13.0, 8.5) {};
    \node (T12T) at (13.0-1.0, 8.5-.7) {};
    \node (T12M) at (13.0-2.0, 8.5-.3) {};
    \node (T12TM) at (13.0-3.0, 8.5-1.0) {};
    \node (T13) [label=right:$T_{1,3}$] at (13.0, 7.5) {};
    \node (T13T) at (13.0-1.0, 7.5-.7) {};
    \node (T13M) at (13.0-2.0, 7.5-.3) {};
    \node (T13TM) at (13.0-3.0, 7.5-1.0) {};
    \node (T1H) [draw=none, fill=none, scale=0.1] at (13.0, 7.0) {};
    \node (T1HT) [draw=none, fill=none, scale=0.1] at (13.0-1.0, 7.0-.7) {};
    \node (T1HM) [draw=none, fill=none, scale=0.1] at (13.0-2.0, 7.0-.3) {};
    \node (T1HTM) [draw=none, fill=none, scale=0.1] at (13.0-3.0, 7.0-1.0) {};
    \node (T1e) [label=right:$T_{1,\infty}$] at (13.0, 5.5) {};
    \node (T1eT) at (13.0-1.0, 5.5-.7) {};
    \node (T1eM) at (13.0-2.0, 5.5-.3) {};
    \node (T1eTM) at (13.0-3.0, 5.5-1.0) {};

    \node (V) [label=below:$V$] at (13.0-3.5,3.7) {};
    \node (VT0) at (13.0-4.5,3.2) {};
    \node (VT1) at (13.0-2.0,3.2) {};
    \node (VT) at (13.0-3.0,2.7) {}; 

    \foreach \from/\to in {
          T02/T02T, T02/T02M, T02T/T02TM, T02M/T02TM,
          T03/T03T, T03/T03M, T03T/T03TM, T03M/T03TM,
          T0e/T0eT, T0e/T0eM, T0eT/T0eTM, T0eM/T0eTM,
          T12/T12T, T12/T12M, T12T/T12TM, T12M/T12TM,
          T13/T13T, T13/T13M, T13T/T13TM, T13M/T13TM,
          T1e/T1eT, T1e/T1eM, T1eT/T1eTM, T1eM/T1eTM,
          T02/T03, T02T/T03T, T02M/T03M, T02TM/T03TM,
          T03/T0H, T03T/T0HT, T03M/T0HM, T03TM/T0HTM,
          T12/T13, T12T/T13T, T12M/T13M, T12TM/T13TM,
          T13/T1H, T13T/T1HT, T13M/T1HM, T13TM/T1HTM,
          T0eM/AT0, T0eTM/AT,
          T1eM/VT1, T1eTM/VT,
          A/AT0, A/AT1, AT0/AT, AT1/AT,
          V/VT0, V/VT1, VT0/VT, VT1/VT,
          A/AV, AT0/T0AV, AT1/T1AV, AT/TAV,
          V/AV, VT0/T0AV, VT1/T1AV, VT/TAV,
          T02TM/SM, T12TM/SM,
          T0/T02, T0M/T02M, T/T02T, TM/T02TM,
          T1/T12, T1M/T12M, T/T12T, TM/T12TM}
    \draw [-] (\from) -- (\to);
    \path (M) edge [out=215, in=90] (A);
    \path (T1M) edge [out=215, in=90] (AT1);
    \path (M) edge [out=325, in=90] (V);
    \path (T0M) edge [out=325, in=90] (VT0);

    \foreach \from/\to in {
          T0H/T0e, T0HT/T0eT, T0HM/T0eM, T0HTM/T0eTM,
          T1H/T1e, T1HT/T1eT, T1HM/T1eM, T1HTM/T1eTM}
    \draw [dotted] (\from) -- (\to);

    \draw [dashed,smooth] plot coordinates{
      ($(T02M) +(-0.5,2)$)
      ($(T02M) +(0.5,0.5)$)
      ($ (SM)  +(0.2,0.3)$)
      ($(T12M) +(-0.5,0.5)$)
      ($(T12M) +( 0.5,2)$)
      };
    \draw [dashed,smooth] plot coordinates{
      ($(AT)    +(0.2,-1.0)$)
      ($(T0eTM) +(0.4, 0.5)$)
      ($ (SM)  +( 0.2,-0.2)$)
      ($(T1eTM)+(-0.4, 0.5)$)
      ($(VT)   +(-0.2,-1.0)$)
      };
    \draw [dashed,smooth] plot coordinates{
      ($(T0AV) +(-0.2,-1.0)$)
      ($(T0L)  +(-0.2, 0.5)$)
      ($ (SM)  +( 0.2,-0.3)$)
      ($(T1L)  +( 0.2, 0.5)$)
      ($(T1AV) +( 0.2,-1.0)$)
      };

    \node[fill=none,scale=2] at ($ (T0) + (-2,0.5) $) {\textbf{Section~\ref{section:introduction}}};
    \node[fill=none,scale=2] at ($ (T02) + (0.5,1) $) {\textbf{Section~\ref{section:T02}}};
    \node[fill=none,scale=2] at ($ (AT) + (-1,-0.5) $) {\textbf{Section~\ref{section:T02}}};
    \node[fill=none,scale=2] at ($ (T0AV) + (-1,-0.5) $) {\textbf{Sec.\ \ref{section:lambda}}};
    \node[fill=none,scale=2] at ($ (TAV) + (0,-0.5) $) {\textbf{Section~\ref{section:L}}};
  \end{tikzpicture}
}
\newcommand{\CLA}{\mathcal{L}_A}
\newcommand{\CBA}{\mathcal{B}_A}

\begin{document}
\title{Dichotomy on intervals of strong partial Boolean clones}


\author[Schölzel]{Karsten Schölzel}



\begin{abstract}
The following result has been shown recently in the form of a dichotomy: 
For every total clone $C$ on $\2 := \{0,1\}$, the set $\intervalD{C}$ of all partial clones on $\2$ 
whose total component is $C$, is either finite or of continuum cardinality. In this paper we show
that the dichotomy holds, even if only strong partial clones are considered, i.e., partial clones
which are closed under taking subfunctions: 
For every total clone $C$ on $\2$, the set $\intervalStr{C}$ of all strong partial clones on $\2$ 
whose total component is $C$, is either finite or of continuum cardinality.
\end{abstract}

\maketitle

\section{Introduction} \label{section:introduction}

Let $A$ be an arbitrary finite set. In the case we deal with Boolean clones we have $A = \2 := \{0,1\}$.

A function $f: A^n \to A$ is called a total function on $A$. A function $f: S \to A$ with $S \subseteq A^n$ is called
partial function on $A$ and we denote the domain by $\dom f := S$. The set $\OA$ is the set of all total functions on $A$, and
$\PA$ is the set of all partial functions on $A$.

The function $e_i^n: A^n \to A$ defined by $e_i^n(x_1,\dots,x_n) := x_i$ is called the $n$-ary \emph{projection} onto
the $i$-th coordinate. For each $a \in A$ the function $c_a^n: A^n \to A$ is defined as $c_a(\xx) = a$ for all $\xx \in A^n$.

  Let $f \in \PA$ be $n$-ary and let $g_1,\dots,g_n \in \PA$ be $m$-ary.
  The \emph{composition} $F := f(g_1,\dots,g_n)$ is an $m$-ary partial function defined by 
  \[
    F(x_1,\dots,x_m) := f(g_1(x_1,\dots,x_m),\dots,g_n(x_1,\dots,x_m))
  \]
  and
  \[\dom F := \left\{ x \in \bigcap_{i=1}^n \dom g_i \,\middle|\, (g_1(x),\dots,g_n(x)) \in \dom f \right\}.\]

$C \subseteq \PA$ is called a partial clone if it is composition closed and contains the projections.
If additionally $C \subseteq \OA$ then $C$ is a total clone.

Let $f,g \in \PA$. Then $f$ is a restriction (or subfunction) of $g$ if $\dom f \subseteq \dom g$ and 
$f(x) = g(x)$ for all $x \in \dom f$, short $f \leq g$.
Let $X \subseteq \PA$. Then the set $\Str{X} \subseteq \PA$ is defined by
\[ \Str{X} := \{ f \in \PA \mid \exists g \in X: f \leq g \}. \]
If $X = \Str{X}$ then $X$ is called \emph{strong}, or \emph{restiction closed}. That means, that $X$
contains every restriction of every of its functions,
i.e., $f \in C$ for every $f \in \PA$ and $g \in C$ with $f \leq g$.

Let $\RelsA[h]$ be the set of all $h$-ary relations on $A$ for some $h \geq 1$, i.e., 
$\RelsA[h] := \{ X \mid X \subseteq A^h \}$. Furthermore, let $\RelsA := \bigcup_{h \geq 1} \RelsA[h]$.

Let $\varrho \in \RelsA[h]$, and $f: S \to A$ with $S \subseteq A^n$ an $n$-ary partial function.
Then $f$ preserves $\varrho$ iff $f(M) \in \varrho$ for any $h\times n$ matrix $M = (m_{ij})$ whose  
rows belong to the domain of $f$, i.e. $(m_{i1},\dots,m_{in}) \in \dom f$ for all $i$, and whose columns belong to $\varrho$. 

Let $\pPol R$ be the set of all partial functions preserving every relation $\varrho \in R$. 
Let $\Pol R := (\pPol R ) \cap \OA$ the set of all total functions preserving every relation $\varrho \in R$.

There at least three different types of intervals which we consider here.
Let $C$ be a total clone of $\OA$.
Then we can define the three intervals $\intervalD{C}$, $\intervalStr{C}$, and $\intervalStrUp{C}$ by
\begin{align*}
  \intervalD{C} & := \{ X \subseteq \PA \mid \textrm{$X$ partial clone}, C = X \cap \OA \} \\
  \intervalStr{C} & := \{ X \subseteq \PA \mid \textrm{$X$ strong partial clone}, C = X \cap \OA \} \\ 
  \intervalStrUp{C} & := \{ X \subseteq \PA \mid \textrm{$X$ strong partial clone}, C \subseteq X \} \\ 
               & = \bigcup_{\substack{D \textrm{ total clone} \\ C \subseteq D}} \intervalStr{D} 
\end{align*}
Clearly, $\intervalStr{C} \subseteq \intervalD{C}$ holds.



The following total Boolean clones are needed in this paper, and every other total Boolean clone can be written
as the intersection of some of these.
\begin{eqnarray*}
T_a & = & \Pol \{a\} \textnormal{ for } a \in \{0,1\} \\
T_{a,\mu} & = & \Pol \left(\{0,1\}^\mu \setminus \{ (b,\dots,b) \}\right) \textnormal{ for } b \in \{0,1\}, b \neq a \\
T_{a,\infty} & = & \bigcap_{\mu \geq 2} T_{a,\mu} \textnormal{ for } a \in \{0,1\} \\ 
M & = & \Pol \begin{pmatrix} 0 & 0 & 1 \\ 0 & 1 & 1 \end{pmatrix} \\
  && \textnormal{(set of all monotone functions)} \\
  S & = & \Pol \begin{pmatrix} 0 & 1 \\ 1 & 0 \end{pmatrix} \\
    && \textnormal{(set of all self-dual functions)} \\
  L & = & \Pol \left\{ (x,x,y,y), (x,y,x,y), (x,y,y,x) \mid x,y \in \{0,1\} \right\} \\
  && \textnormal{(set of all linear functions)} \\
\Lambda & = & \clone \left\{ \land, c_0, c_1 \right\} \\
V & = &       \clone \left\{ \lor,  c_0, c_1 \right\} \\
\Omega_1 & = & \clone \left(\OpBool[1]\right)
\end{eqnarray*}

\begin{figure}[ht]
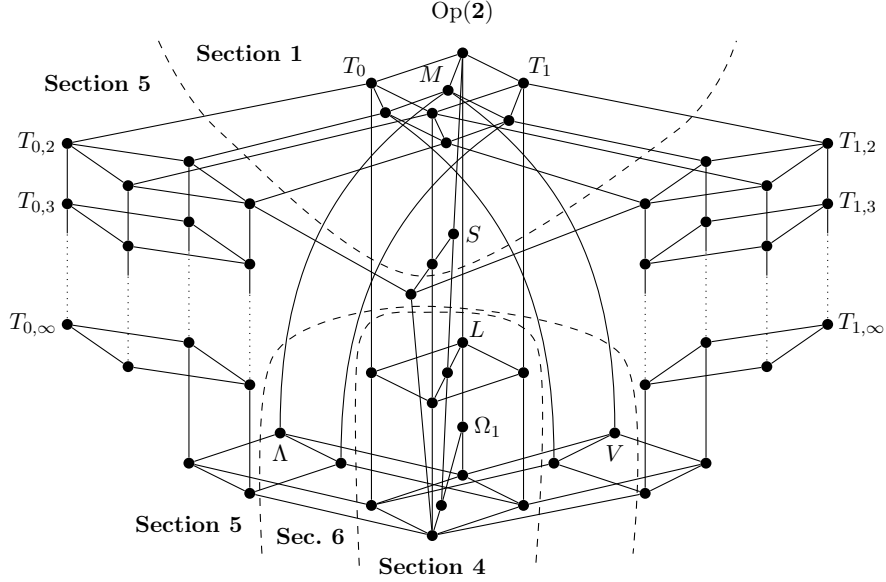

\begin{center}
  \PostsLattice{0.8}
\end{center}
\caption{Post's lattice 
(with indications in which sections the corresponding intervals are handled)}
\label{figure:PostsLattice}
\end{figure}

In \cite{Alekseev-Voronenko:1994,Strauch:1995:MT0T1,Strauch:1996:ST0T1,Strauch:1997a,LauSchoelzel:2010ismvl2}
the finite intervals $\intervalD{C}$ have been determined and in \cite{HaddadSimons:2003,FugereHaddad:1998} the
finite intervals of the form $\intervalStrUp{C}$. These results can be assembled into the following theorem.
The column with the sizes for $\intervalStr{C}$ can be easily deduced from the sizes $\intervalStrUp{C}$ and
Post's lattice. The finite intervals $\intervalStrUp{M \cap T_0 \cap T_1}$ and $\intervalStrUp{S \cap T_0 \cap T_1}$
are displayed in Appendix~\ref{section:finiteIntervalsDrawings}.

\begin{theorem} \label{theorem:finiteIntervals}
  Let $C$ be a total Boolean clone with
  \[
    T_0 \cap T_1 \cap M \subseteq C \textnormal{ or } T_0 \cap T_1 \cap S \subseteq C. 
  \]
  Then $\intervalD{C}$ and $\intervalStrUp{C}$ are finite sets. Furthermore it holds that
  \[
\begin{array}{|r|c|c|c|}
\hline
C & |\intervalD{C}| & |\intervalStrUp{C}| & |\intervalStr{C}| \\
\hline
\OpBool & 3 & 1 & 1 \\
T_a \; (a\in \{0,1\}) & 6 & 2 & 1 \\
M & 6 & 2 & 1 \\
S & 6 & 2 & 1 \\
T_0 \cap T_1 & 30 & 7 & 4 \\
M \cap T_a \; (a\in \{0,1\}) & 15 & 5 & 2 \\
M \cap T_0 \cap T_1 & 101 & 25 & 13 \\
S \cap T_0 \cap T_1 & 380 & 33 & 25 \\
\hline
\end{array}
  \]
\end{theorem}

In \cite{Alekseev-Voronenko:1994, Strauch:1997b} it was shown that the intervals $\intervalD{C}$ for subclones 
$C \subseteq B$ with $B \in \{L,\Lambda, V, T_{0,\infty}, T_{1,\infty}\}$ have the size of the continuum.
Then in \cite{LauSchoelzel:2010ismvl2} the remaining intervals were determined to be infinite. The authors
of \cite{CouceiroHaddadSchoelzelWaldhauser:2013:ISMVL} then finished the determination of the intervals
of the form $\intervalD{C}$ to yield the following theorem.

\begin{theorem}
  Let $C$ be a total Boolean clone such that $C \subseteq B$ and $B \in \{L,\Lambda, V, T_{0,2}, T_{1,2}\}$.
  Then the set $\intervalD{C}$ has the cardinality of the continuum.
\end{theorem}

As stated in \cite{CouceiroHaddadSchoelzelWaldhauser:2013:ISMVL} this yields a dichotomy on the size of
the intervals $\intervalD{C}$ for Boolean clones $C$.

\begin{theorem}
  Let $C$ be a total Boolean clone.

  Then $\intervalD{C}$ is either finite or has the cardinality of the continuum. Furthermore,
  $\intervalD{C}$ is finite if and only if $M \cap T_0 \cap T_1 \subseteq C$ or $S \cap T_0 \cap T_1 \subseteq C$.
\end{theorem}

The aim of this paper is to show that this result can be strengthend in the sense that only strong 
partial clones are considered. That means $\intervalD{C}$ will be replaced by $\intervalStr{C}$ in the 
statement of the last theorem. Since $\intervalStr{C} \subseteq \intervalD{C}$ for every Boolean clone $C$,
we already have that $\intervalStr{C}$ is finite if 
$M \cap T_0 \cap T_1 \subseteq C$ or $S \cap T_0 \cap T_1 \subseteq C$.
Thus we will show that the interval $\intervalStr{C}$ has the cardinality of the continuum for all 
other Boolean clones $C$.

Although we focus on the case of Boolean clones, there have been some investigations into
the general case with $|A| \geq 2$, for example \cite{Haddad2007257} and \cite{Had-L-R:2006}.
Some of these results will be extended with the help of Lemma \ref{lemma:diffconstants}.

\section{Further definitions} \label{section:furtherDefinitions}

For some natural numbers $n,m \in \IN$ with $n \leq m$ we define the sets $[n,m] := \{n,n+1,\dots,m\}$,
and $[n] := [1,n]$. Tuples will be written with boldface small letters, and with the exception of $\2 = \{0,1\}$ 
a small boldface letter signifies a tuple. For a tuple $\xx := (x_1,\dots,x_n) \in A^n$ 
we define the set of its entries by $[\xx] := \{x_1,\dots,x_n\}$, and let $|\xx| := |[\xx]|$. 
For $I \subseteq [n]$ we let $\xx_I := \{ x_i \mid i \in I \}$. For $\ii = (i_1,\dots,i_l) \in [n]^l$
with $l \in \IN$ we define $\xx_{\ii} := (x_{i_1}, \dots, x_{i_l}) \in A^l$.
We will often use the two special tuples $\0 := (0,\dots,0)$ and $\1 := (1,\dots,1)$; 
the length of these tuples can be deduced from the context.

\subsection{Romov's definability lemma}

The statement of Theorem~\ref{theorem:definability} proven by Romov in \cite{Romov:1981} 
gives a nice characterization of the constructability
of relations in the co-clone of a strong partial clone. This enables us to prove 
the Theorems~\ref{theorem:T02:nonconstructible} and \ref{theorem:Lambda:nonconstructible} just
with relational methods.

The relation $\rho \in \RelsA[h]$ is called \emph{irredundant} iff it fulfills the following two conditions:
\begin{enumerate}[label=(\roman*)]
  \item $\rho$ has no duplicate rows, i.e., for all $i,j$ with $1 \leq i < j \leq h$, there is a tuple 
    $(a_1,\dots,a_h) \in \rho$ with $a_i \neq a_j$;
  \item $\rho$ has no fictitious coordinates, i.e., there is no $i \in \{1,\dots,h\}$, such that $(a_1,\dots,a_h) \in \rho$
    implies $(a_1,\dots,a_{i-1},x,a_{i+1},\dots,a_h) \in \rho$ for all $x \in A$.
\end{enumerate}

For a relation $\sigma \in \RelsA[h]$ we define $\Arity \sigma := h$.

\begin{theorem} \label{theorem:definability} 
  Let $\Sigma \subseteq \RelsA$ and $\rho \in \RelsA[t]$ be relations. Furthermore let $\rho$ be irredundant.
  Then 
  \[
  \bigcap_{\sigma \in \Sigma} \pPol \sigma \subseteq \pPol \rho
  \]
  iff
  there are some $\gamma_\sigma \subseteq [t]^{\Arity \sigma}$ for all $\sigma \in \Sigma$
  such that
  \[
  \rho = \{ \xx \in A^t \mid \xx_\ii \in \sigma \text{ for all } \ii \in \gamma_\sigma \text{ and } \sigma \in \Sigma \}
  \]
  and
  \[
  [t] = \bigcup_{\sigma \in \Sigma} \bigcup_{\ii \in \gamma_\sigma} [\ii].
  \]
\end{theorem}

\section{Classes of partial functions} \label{section:partialClasses}

The aim of this section is the introduction of classes of partial functions (or shorter: partial classes)
similar to the ideas presented by Harnau in 
\cite{Harnau:19857:Relationenpaare:I,Harnau:19857:Relationenpaare:II,Harnau:19857:Relationenpaare:III}
for total clones. This concepts will prove fruitful in the extension of Theorem~8~\cite{LauSchoelzel:2010ismvl2}
as shown in Lemma~\ref{lemma:strongdownwards}. Since we do not need the full power of the Galois connection
presented by Harnau we will only prove statements about partial classes relevant to this paper.

For the definition of a partial class we need to define the following Maltsev-operations 
$\zeta$, $\tau$, $\Delta$, $\nabla$, and $\star$.
Let $f \in \PA[n]$ and $g \in \PA[m]$.
Then we define
\begin{align*}
  (\zeta f)(x_1,\dots,x_n) & := f(x_2,x_3,\dots,x_n,x_1), \\
  (\tau f)(x_1,\dots,x_n)  & := f(x_2,x_1,x_3,\dots,x_n), \\
  (\Delta f)(x_1,\dots,x_{n-1}) & := f(x_1,x_1,x_2,\dots,x_{n-1}), \\
  & \zeta f = \tau f = \Delta f = f \text{ if } n = 1, \\
  (\nabla f)(x_1,\dots,x_{n+1}) & := f(x_2,\dots,x_{n+1}), \\
  (f \star g)(x_1,\dots,x_{n+m-1}) & := f(g(x_1,\dots,x_m),x_{m+1},\dots,x_{n+m-1}).
\end{align*}

\begin{definition}
  Let $X \subseteq \PA$. Then $X$ is called a \emph{partial class} if it closed 
  under the operations $\star$, $\zeta$, $\tau$, $\nabla$, and $\Delta$.
\end{definition}

\begin{lemma} \label{lemma:intersectingClasses}
  Let $X,Y \subseteq \PA$ be two partial classes. Then $X \cap Y$ is also a partial class.
\end{lemma}

The partial classes containing the projections are exactly the partial clones. 

If $X, Y \subseteq \PA$, then we define the set $X \star Y \subseteq \PA$ by
\[ X \star Y := \{ f \star g \mid f \in X, g \in Y \}. \]

\subsection{Relation pairs}

Similar to the work done by Harnau in 
\cite{Harnau:19857:Relationenpaare:I,Harnau:19857:Relationenpaare:II,Harnau:19857:Relationenpaare:III}
we introduce relation pairs to characterize strong partial classes.

For each $h \geq 1$ let $\cPairsA[h]$ be the set of all pairs $(\rho, \rho')$ 
with $\rho' \subseteq \rho \subseteq A^h$,
and $\cPairsA := \bigcup_{h \geq 1} \cPairsA[h]$.

Let $(\rho, \rho') \in \cPairsA[h]$ for some $h \geq 1$, and $f \in \PA[n]$ for some $n \geq 1$. Then $f$ \emph{preserves}
the relation pair $(\rho,\rho')$, if for all matrices $M$ with columns in $\rho$, and lines in $\dom f$
the tuple $f(M)$ belongs to $\rho'$. We write $f \in \cPol (\rho,\rho')$, or $(\rho,\rho') \in \cInv f$.

If $\rho = \rho'$ then the preservation of the relation pair $(\rho,\rho')$ 
coincides with the preservation of the relation $\rho$,
i.e., $\cPol (\rho,\rho) = \pPol \rho$.

If $X \subseteq \PA$, and $Q \subseteq \cPairsA$, then we define
\begin{align*}
  \cPol Q & := \bigcap_{q \in Q} \cPol q, \\
  \cInv X & := \bigcap_{f \in X} \cInv f.
\end{align*}

\begin{lemma} \label{lemma:starComposing}
  Let $f \in \cPol (\rho,\rho')$ and $g \in \cPol (\sigma, \sigma')$ 
  with $\sigma' \subseteq \rho \subseteq \sigma$.
  
  Then $f \star g \in \cPol (\rho,\rho')$.
\end{lemma}

\begin{proof}
  Let $f \in \cPol[n] (\rho,\rho')$ and $g \in \cPol[m] (\sigma, \sigma')$ 
  with $\sigma' \subseteq \rho \subseteq \sigma \in \RelsA[h]$.
  Let $M$ be an $(h,m+n-1)$-matrix with columns $\xx_1,\dots,\xx_{m+n-1} \in \rho$, and
  rows $\yy_1,\dots,\yy_h \in \dom f\star g$. Let $\yy'_j := (\yy_j)_{(1,\dots,m)}$ for each $j \in [h]$.
  
  Then $\yy'_1,\dots,\yy'_h \in \dom g$ by the definition of $\star$, and 
  $\xx_1,\dots,\xx_m \in \sigma$. Thus $\xx := g(\xx_1,\dots,\xx_m) \in \sigma' \subseteq \rho$.
  From this $(f\star g)(\xx_1,\dots,\xx_{m+n-1}) = f(\xx,\xx_{m+1},\dots,\xx_{m+n-1}) \in \rho'$
  and thus $f \star g \in \cPol (\rho,\rho')$.
\end{proof}

\begin{lemma} \label{lemma:cPolqIsClass}
  Let $q \in \cPairsA$. Then $\cPol q$ is a non-empty strong partial class of $\PA$.
\end{lemma}

\begin{proof}
  Let $(\rho,\rho') := q \in \cPairsA$.

  We first show that $\cPol (\rho,\rho')$ is a partial class. Let $f,g \in \cPol (\rho,\rho')$. 
  
  It is easy to see that $\zeta f, \tau f, \Delta f, \nabla f \in \cPol (\rho,\rho')$.
  From Lemma~\ref{lemma:starComposing} with $\sigma = \rho$ and $\sigma' = \rho'$ follows 
  $f \star g \in \cPol (\rho,\rho')$. Thus $\cPol (\rho,\rho')$ is a partial class of $\PA$.

  We now want to show that $\cPol (\rho,\rho')$ is strong.
  Let $f \in \cPol (\rho,\rho')$ and $g \leq f$, and assume to the contrary that $g \notin \cPol (\rho,\rho')$.
  Then there is a matrix $M$ with columns $\xx_1,\dots,\xx_n \in \rho$ and rows $\yy_1,\dots,\yy_h \in \dom g$,
  such that $g(M) \notin \rho'$.
  Since $\dom g \subseteq \dom f$ and $f(M) = g(M) \notin \rho'$. Thus $f \notin \cPol (\rho,\rho')$
  contradicting the assumption. Thus $\cPol (\rho,\rho')$ is strong.

  It is non-empty since the partial function $c_\emptyset$ with empty domain perserves any relation pair $q$.
\end{proof}

\begin{lemma}
  Let $Q \subseteq \cPairsA$. Then $\cPol Q$ is a non-empty strong partial class of $\PA$.
\end{lemma}

\begin{proof}
  By Lemma~\ref{lemma:cPolqIsClass} we have that $\cPol q$ is a strong partial class for all $q \in Q$.
  Then by Lemma~\ref{lemma:intersectingClasses} and the definition of $\cPol Q$, 
  we see that $\cPol Q$ is a partial class. Furthermore, the intersection of two strong sets is also strong.
  It is non-empty since $c_\emptyset \in \cPol Q$.
\end{proof}

\begin{remark}
  It is possible to show, that for every non-empty strong partial class $X \subseteq \PA$, 
  there is some $Q \subseteq \cPairsA$ with $X = \cPol Q$. Since this and other further properties
  of the operators $\cPol$ and $\cInv$ are not needed in this paper, they will not be proven here.
\end{remark}

\begin{lemma} \label{lemma:emptyConsequentNoTotalFunction}
  Let $\rho \in \RelsA$ with $\rho \neq \emptyset$. 
  
  Then $\cPol (\rho,\emptyset) \cap \OA = \emptyset$.  
\end{lemma}

\begin{proof}
  Let $f \in \OA[n]$, $\rho \in \RelsA[h]$, and $\xx \in \rho$. 
  Let $M$ be the matrix formed by $n$-fold repetition of the column $\xx$. 
  Let the rows of $M$ be called $\yy_1,\dots,\yy_h$. Clearly, $\yy_i \in \dom f$ for all $i \in [h]$ since
  $f$ is a total function. But $f(M) \notin \emptyset$, and thus $f \notin \cPol (\rho,\emptyset)$.
\end{proof}

\begin{lemma}
  Let $\rho \in \RelsA$, $f \in \PA$ and $g \in \cPol (\rho, \emptyset)$.

  Then $f \star g \in \cPol (\rho,\emptyset)$.
\end{lemma}

\begin{proof}
  Let $\rho \in \RelsA[h]$, $f \in \PA[n]$ and $g \in \cPol[m] (\rho, \emptyset)$.

  If $\rho = \emptyset$, then $\cPol (\rho,\emptyset) = \cPol(\emptyset,\emptyset) = \PA$. 
  Thus $f \star g \in \cPol (\rho,\emptyset)$.

  Let $\rho \neq \emptyset$. Assume to the contrary, that $f \star g \notin \cPol (\rho,\emptyset)$.
  Then there is a matrix $M$ with columns $\xx_1,\dots,\xx_{m+n-1} \in \rho$, 
  and rows $\yy_1,\dots,\yy_h \in \dom (f \star g)$. 
  We can now look at the matrix $M'$ formed by the first $m$ columns, and with rows $\yy_1',\dots,\yy_h'$.
  Then $\yy_i \in \dom (f \star g)$ implies $\yy_i' \in \dom g$ for all $i \in [h]$. But
  since $\xx_1,\dots,\xx_m \in \rho$ we get $g \notin \cPol (\rho, \emptyset)$ in contradiction
  to the assumption.
\end{proof}

\begin{corollary}
  Let $X \subseteq \PA$ and $\rho \in \RelsA$.

  Then $X \star \cPol (\rho,\emptyset) \subseteq \cPol (\rho,\emptyset)$.
\end{corollary}

The following corollary follows from Lemma~\ref{lemma:starComposing}.
\begin{corollary}
  Let $\rho \in \RelsA$.

  Then $\cPol (\rho, \emptyset) \star \pPol \rho \subseteq \cPol (\rho,\emptyset)$.
\end{corollary}

The last two corollaries can now be combined into the final statement of this subsection.
\begin{corollary} \label{corollary:emptysetIsVeryHungry}
  Let $\rho \in \RelsA$, $T := \cPol (\rho,\emptyset)$ and $D \subseteq \pPol \rho$.

  Then $T \star D \subseteq T$ and $D \star T \subseteq T$.
\end{corollary}

\subsection{Classes to intervals}

In the proof that the interval $\intervalStr{D}$ are of continuum cardinality for some
total clone $D$, we try to make as few constructions as possible. This can be achieved if
we find some clone $C$ with $D \subseteq C$, construct a set $I \subseteq \intervalStr{C}$
of continuum cardinality, and then find restrictions of the partial clones in $I$, such that
these restricted partial clones lie in $\intervalStr{D}$, and $I$ does not collapse.

For this purpose we prove a stronger version of Theorem~8~\cite{LauSchoelzel:2010ismvl2} as follows.

\begin{lemma} \label{lemma:strongdownwards}
Let $C$ and $D$ be clones of $\OA$ with $D \subseteq C$, $T$ a strong partial class of $\PA$,
and $I \subseteq \intervalStr{C}$, such that the following conditions hold
\begin{enumerate}[label=(\roman*)]
  \item \label{enum:strongdownwards:totalempty}
    $T \cap \OA \subseteq D$,
  \item \label{enum:strongdownwards:QiHungry}   
    $T \star \Str{D} \subseteq \Str{D} \cup T$, and  $\Str{D} \star T \subseteq \Str{D} \cup T$,
  \item \label{enum:strongdownwards:different}
    $X \cap T \neq Y \cap T$ for all $X,Y \in I$ with $X \neq Y$.
\end{enumerate}
Then
  \[
    |\intervalStr{D}| \geq |I|.
  \]
\end{lemma}

\begin{proof}
  For each $X \in I$ we define $X_D$ by 
  \[
  X_D := \Str{D} \cup (X \cap T).
  \]
  We let $I_D := \{ X_D \mid X \in I \}$, and show that $I_D \subseteq \intervalStr{D}$.
  By \ref{enum:strongdownwards:different} we have that $|I_D| \geq |I|$.
  
  Let $X \in I$ be arbitrary. By \ref{enum:strongdownwards:totalempty} we have that 
  \begin{align*}
    X_D \cap \OA & = ( \Str{D} \cup (X \cap T)) \cap \OA \\
    & = (\underbrace{\Str{D} \cap \OA}_{D}) \cup (\underbrace{X \cap (T \cap \OA)}_{\subseteq D}) \\
    & = D.
  \end{align*}
  Thus we only have to show that $X_D$ is a strong partial clone.

  Since $\Str{D}$, $X$, and $T$ are strong partial classes, we see that $\Str{X_D} = X_D$,
  and that $X_D$ is closed with respect to $\zeta$, $\tau$, $\nabla$ and $\Delta$. Furthermore,
  $X_D$ contains the projections, since $\Str{D} \subseteq X_D$, and $D$ is a clone.

  It remains to show that $X_D$ is closed with respect to $\star$. 
  Let $f,g \in X_D$. We want to show that $f \star g \in X_D$.
  Since $D \subseteq C \subseteq X$, $X \cap T \subseteq X$ and $X$ is a partial clone, 
  we have $f \star g \in X$.

  There are several cases:
  \begin{itemize}
    \item
      $f,g \in \Str{D}$. Then $f \star g \in \Str{D} \subseteq X_D$, since $\Str{D}$ is a strong partial clone.
    \item
      $f,g \in X \cap T$. Then $f \star g \in X \cap T \subseteq X_X$, since $X \cap T$ is a strong partial class.
    \item
      $f \in \Str{D}$, and $g \in X \cap T$; or
      $g \in \Str{D}$, and $f \in X \cap T$. 
      By \ref{enum:strongdownwards:QiHungry} we have $f \star g \in \Str{D} \cup T$.
      Thus 
      \begin{align*}
        f \star g & \in (\Str{D} \cup T) \cap X \\
        & = ( \Str{D} \cap X) \cup (X \cap T) \\
        & = \Str{D} \cup (X \cap T) \\
        & = X_D.
      \end{align*}
  \end{itemize}
  Thus $X_D$ is a strong partial clone with $X_D \cap \OA = D$. This implies $X_D \in \intervalStr{D}$.
  Therefore $I_D \subseteq \intervalStr{D}$, and consequently $|\intervalStr{D}| \geq |I|$.
\end{proof}

One example of the strong partial class $T$ needed in the preceding lemma is the partial class
$\cPol(\{0\},\emptyset)$ of all partial functions not defined on $(0,\dots,0)$. This was implicitly
used for example in \cite{LauSchoelzel:2010ismvl2} and \cite{CouceiroHaddadSchoelzelWaldhauser:2013:ISMVL}.

Each of the sets $I$ defined in this paper will be indexed by the subsets of a countable infinite set 
$N \subseteq \IN$. As such the set $I$ has the same cardinality as the powerset of $\IN$, 
which has the cardinality of the continuum, and therefore $I$ is of continuum cardinality.

\subsection{Subclones missing a constant}

First we use Lemma \ref{lemma:strongdownwards} in a general setting, involving two clones $C$ and $D$ in $\OA$ with
$D \subseteq C$ and $c_a \in C \setminus D$ for some $a \in A$.
For a partial function $f \in \PA[n]$ and some $a \in A$ we define the $(n+1)$-ary partial function $f_a \in \PA$
by
\begin{align*}
  \dom f_{a} & := \{ (a,\xx) \mid \xx \in \dom f \}, \\
  f_{a}(a,\xx) & := f(\xx) \text{ for all } \xx \in \dom f.
\end{align*}

\begin{lemma} \label{lemma:constants}
  Let $C \subseteq \OA$ be a clone with $c_a \in C$, and $X \in \intervalStr{C}$. 
  Then $f \in X$ if and only if $f_a \in X$. 
\end{lemma}

\begin{proof}
  Assume $f \in X$. Then $f_a \leq \nabla f \in X = \Str{X}$, and thus $f \in X$.

  Now assume that $f_a \in X$. Additionally, we have $c_a \in C \subseteq X$. Thus 
  $f = \Delta (f_a \star c_a) \in X$.
\end{proof}

\begin{lemma} \label{lemma:noconstant}
  Let $D \subseteq \OA$ be a clone with $c_a \notin D$. Then there is some $\rho \in \Inv D$
  with $(a,\dots,a) \notin \rho$.
\end{lemma}

\begin{proof}
  Assume to the contrary, that $(a,\dots,a) \in \rho$ for all $\rho \in \Inv D$. Then $c_a \in \Pol \rho$
  for all $\rho \in \Inv D$, and thus $c_a \in D$. Contradiction.
\end{proof}

\begin{lemma} \label{lemma:diffconstants}
  Let $C,D \subseteq \OA$ be clones with $c_a \in C \setminus D$ and $D \subseteq C$.
  Then $|\intervalStr{D}| \geq |\intervalStr{C}|$.
\end{lemma}

\begin{proof}
  By Lemma \ref{lemma:noconstant} there is some relation $\rho$ with $(a,\dots,a) \notin \rho$ and
  $D \subseteq \Pol \rho$. Let $T := \cPol(\rho,\emptyset)$, and $I := \intervalStr{C}$. 
  We want to use Lemma \ref{lemma:strongdownwards}.
  
  Since $T \cap \OA = \emptyset \subseteq D$ we have condition \ref{enum:strongdownwards:totalempty}, and
  by Corollary \ref{corollary:emptysetIsVeryHungry} we have condition \ref{enum:strongdownwards:QiHungry}.

  Now we want to show condition \ref{enum:strongdownwards:different}.
  Now let $X,Y \in \intervalStr{C}$ with $X \neq Y$; w.l.o.g. there is some $f \in X \setminus Y$.
  By Lemma \ref{lemma:constants} we have $f_a \in X \setminus Y$. We just need to show that $f_a \in T$.
  
  Assume to the contrary that $f_a \notin T$. Let $f_a$ be $n$-ary, and $\rho$ be $h$-ary. 
  Then there is a matrix $M$ such that
  \begin{itemize}
    \item its row $\xx_1,\xx_2,\dots,\xx_h \in \dom f_a$, and
    \item its columns $\yy_1,\dots,\yy_n \in \rho$.
  \end{itemize}
  By the definition of $f_a$ and choice of $\rho$ we see that $\yy_1 = (a,\dots,a) \notin \rho$.
  This is a contradiction. Thus $f_a \in T$, and consequently $X \cap T \neq Y \cap T$.

  Therefore all conditions of Lemma \ref{lemma:strongdownwards} are fulfilled, and we get 
  $|\intervalStr{D}| \geq |\intervalStr{C}|$.
\end{proof}

This lemma can be applied to the main results of Theorems 10 and 19 in \cite{Haddad2007257}.
Let $\CBA$ be the set of all $h$-universal relations ($3 \leq h \leq |A|-1$), and
let $\CLA$ be the set of all prime affine relations on $A$. Then for each $\rho \in \CBA \cup \CLA$
the following properties hold
\begin{itemize}
  \item $\Pol \rho$ is a maximal clone of $\OA$,
  \item $c_a \in \Pol \rho$ for all $a \in A$,
  \item $\intervalStr{\Pol \rho}$ has the cardinality of the continuum.
\end{itemize}
With Lemma \ref{lemma:diffconstants} we obtain the following statement.
\begin{theorem}
  Let $D \subseteq \OA$ a clone with $D \subseteq \Pol \rho$ for some $\rho \in \CBA \cup \CLA$,
  and $c_a \notin D$ for some $a \in A$. Then $\intervalStr{D}$ has the cardinality of the continuum.
\end{theorem}

\section{The subclones of $L$} \label{section:L}

In this section we use the results from \cite{Alekseev-Voronenko:1994} to show that the interval
$\intervalStr{D}$ has continuum cardinality for all clones $D \subseteq L$.

We need to define some functions first as given in \cite{Alekseev-Voronenko:1994}. 
Let $n(k,p) := (2k-1)p+1$, $k \geq 2$ and $p \geq 1$. 
Define the $n(k,p)$-ary partial function $\tau_p^k$ by 
\begin{align*}
  \dom \tau_p^k & := \{\1\} \cup \{ \xx \in \2^{n(k,p)} \mid \#_1 \xx \leq p \}, \\
  \tau_p^k(\xx) & := \begin{cases}
    1 & \text{if } \xx = \1, \\
    0 & \text{if } \xx \in \dom \tau_p^k \setminus \{\1\}.
  \end{cases}
\end{align*}

We define $p_j$ by $p_1 := 1$ and $p_j := n(j,p_{j-1})$ for all $j \geq 2$.
Set $\xi_j := \tau_{p_j}^{j+1}$ for all $j \geq 1$.

\begin{lemma}[Лемма 11 \cite{Alekseev-Voronenko:1994}] \label{lemma:xisIndependent}
  Let $j \geq 1$.

  Then $\xi_j \notin [\{\xi_1,\dots,\xi_{j-1},\xi_{j+1},\dots\} \cup \Str{L} ]$.
\end{lemma}

As a consequence we get the following theorem.

\begin{theorem} \label{theorem:Lcontinuum}
  The interval $\intervalStr{L}$ has the cardinality of the continuum.
\end{theorem}

\begin{proof}
  Let $X_J := [\{\xi_j \mid j \in J\} \cup \Str{L}]$ for every $J \subseteq \IN \setminus \{0\}$.
  By Lemma~\ref{lemma:xisIndependent} we see that $X_J \neq X_{J'}$ if $J \neq J'$, and thus
  the set $I := \{ X_J \mid J \subseteq \IN \setminus \{0\} \}$ has the cardinality of the continuum.
  Furthermore, $I \subseteq \intervalStrUp{L}$. Since $L$ is a maximal clone and $|\intervalStr{\OpBool}| = 1$,
  we conclude that $\intervalStr{L}$ has the cardinality of the continuum.
\end{proof}

\begin{lemma} \label{lemma:LT0LT1LScontinuum}
  Let $D \subseteq L$ be a clone with $C \subseteq D$ with $D \in \{ T_0, T_1, S \}$.
  Then $\intervalStr{D}$ has the cardinality of the continuum.
\end{lemma}

\begin{proof}
  We have $c_0,c_1 \in L$, and $c_1 \notin T_0$, $c_0 \notin T_1$, $c_0 \notin S$. Thus
  Lemma \ref{lemma:diffconstants} is applicable with $C = L$, and by \ref{theorem:Lcontinuum}
  follows that $\intervalStr{D}$ has the cardinality of the continuum.
\end{proof}

\subsection{The remaining two subclones of $L$}

The only two subclones of $L$ not covered yet are $C_{01} := [c_0,c_1]$ and $\Omega_1 := [\OpBool[1]]$.
Let $\rho_C$, $\rho_1$ and $\rho_L$ be three 4-ary relations defined as
\[ \arraycolsep=2.0pt
  \rho_C := 
  \begin{pmatrix} 
    0 & 0 & 0 & 1 \\ 
    0 & 0 & 1 & 1 \\ 
    0 & 1 & 0 & 1 \\ 
    0 & 1 & 1 & 1 
  \end{pmatrix} \quad
  \rho_1 := 
  \begin{pmatrix} 
    0 & 0 & 0 & 1 & 1 & 1 \\ 
    0 & 0 & 1 & 0 & 1 & 1 \\ 
    0 & 1 & 0 & 1 & 0 & 1 \\ 
    0 & 1 & 1 & 0 & 0 & 1 
  \end{pmatrix} \quad
  \rho_L := 
  \begin{pmatrix} 
    0 & 0 & 0 & 0 & 1 & 1 & 1 & 1 \\ 
    0 & 0 & 1 & 1 & 0 & 0 & 1 & 1 \\ 
    0 & 1 & 0 & 1 & 0 & 1 & 0 & 1 \\ 
    0 & 1 & 1 & 0 & 1 & 0 & 0 & 1 
  \end{pmatrix}
\]
Although the fact that $\rho_C \subseteq \rho_1 \subseteq \rho_L$ holds, is not used directly,
the similar structure makes the proof of Lemma~\ref{lemma:xipreserve} a bit easier.

As shown by Blochina in \cite{Blochina:1970:Posta} (see also Section~10.2~\cite{Lau:2006})
the relations $\rho_C$, $\rho_1$ and $\rho_L$ characterize the clones $C_{01}$, $\Omega_1$, and $L$,
respectively. That means the following equalities hold:
\begin{align*}
  C_{01} & = \Pol \rho_C, \\ 
  \Omega_1 & = \Pol \rho_1, \\
  L & = \Pol \rho_L.
\end{align*}

\begin{lemma} \label{lemma:xipreserve}
  Let $j \geq 1$. 
  
  Then $\xi_j \in \pPol \rho_1$ and $\xi_j \in \pPol \rho_C$.
\end{lemma}

\begin{proof}
  Let $\rho \in \{ \rho_1, \rho_C \}$.
  Assume to the contrary, that $\xi_j$ does not preserve $\rho$.

  Then there is a matrix $M$ such that
  \begin{itemize}
    \item its rows $\xx_1,\xx_2,\xx_3,\xx_4 \in \dom \xi_j$,
    \item its columns $\yy_1,\dots,\yy_{p_{j+1}} \in \rho$, and
    \item $\zz := (z_1,z_2,z_3,z_4) := (\xi_j(\xx_1),\xi_j(\xx_2),\xi_j(\xx_3),\xi_j(\xx_4)) \notin \rho$.
  \end{itemize}
  We will show that $\zz = \yy_l$ for some $l \in [p_{j+1}]$. 

  Let $K := \{ k \in [4] \mid \xx_k \neq \1 \}$. Clearly, $z_k = 0$ iff $k \in K$.
  Since $\0,\1 \in \rho$, we have that $\zz \notin \{\0,\1\}$ and 
  thus there are $i,i' \in [4]$ with $\xx_i = \1$ and $\xx_{i'} \neq \1$. 
  This implies $1 \leq |K| \leq 3$.

  By the construction of $\xi_j$ each row $\xx_k$ for $k \in K$ has at most $p_j$-many $1$'s. 
  But $\xi_j$ has an arity of
  \[ p_{j+1} = (2(j+1)-1)p_j + 1 = (2j+1)p_j+1 \geq 3p_j + 1. \]
  Thus there is some column $\yy_l$ with
  $(\yy_l)_k = 0$ for all $k \in K$. Furthermore, $(\yy_l)_{k'} = 1$ for all $k' \in [4] \setminus K$.
  Thus $\yy_l = \zz$. But this contradicts $\yy_l \in \rho$ and $\zz \notin \rho$.

  Therefore $\xi_j \in \pPol \rho$.
\end{proof}

\begin{lemma} \label{lemma:C01continuum}
  $\intervalStr{C_{01}}$ has the cardinality of the continuum.
\end{lemma}

\begin{proof}
  Let $X_J := [\{\xi_j \mid j \in J\} \cup \Str{L}]$ for every $J \subseteq \IN \setminus \{0\}$.
  Then the set $I := \{ X_J \mid J \subseteq \IN \setminus \{0\} \}$ has the cardinality of the continuum.

  Let $T := \pPol \rho_C$.
  \begin{itemize}
    \item $T \cap \OpBool = \Pol \rho_C = C_{01}$.
    \item As $\Str{C_{01}} \subseteq T$ we have 
      $T \star \Str{C_{01}} \subseteq T$, and  $\Str{C_{01}} \star T \subseteq T$.
    \item By Lemma~\ref{lemma:xipreserve} and the definition of $I$, we get
      $X \cap T \neq Y \cap T$ for all $X,Y \in I$ with $X \neq Y$.
  \end{itemize}

  Then we apply Lemma~\ref{lemma:strongdownwards}, and yield the result.
\end{proof}

By setting $T = \pPol \rho_1$ in the previous proof we obtain the proof for the following lemma.
\begin{lemma} \label{lemma:Omega1continuum}
  $\intervalStr{\Omega_1}$ has the cardinality of the continuum.
\end{lemma}

Now we can conclude from Theorem~\ref{theorem:Lcontinuum}, and 
Lemmas~\ref{lemma:LT0LT1LScontinuum}, \ref{lemma:C01continuum}, 
\ref{lemma:Omega1continuum}, that the following theorem holds.

\begin{theorem} \label{theorem:subL}
  Let $D \subseteq L$. Then $\intervalStr{D}$ has the cardinality of the continuum.
\end{theorem}

\section{The clone $T_{0,2}$ and its subclones} \label{section:T02}

\newcommand{\RTzerotwo}{R^{0,2}}

In this section we first give an alternative proof for the fact that $\intervalStr{T_{0,2}}$
has the cardinality of the continuum. The relations used are similar to the ones given in 
\cite{CouceiroHaddadSchoelzelWaldhauser:2013:ISMVL}, but the proof here only uses relations.

\subsection{Alternative proof for $\intervalStr{T_{0,2}}$ is continuum} 
The proof given in this section uses some ideas from the proof in \cite{CouceiroHaddadSchoelzelWaldhauser:2013:ISMVL},
but changes the basic building blocks of the relations used. Furthermore, 
while the former proof depended on working with functions, this proof here only deals with relations. 

Let $\rho_{0,2} := \{ (0,0), (0,1), (1,0) \}$. Then we remember that $T_{0,2} = \Pol \rho_{0,2}$.

Let $\RTzerotwo_{C,n}$ and $\RTzerotwo_{K,n}$ be two $n$-ary relations defined by
\begin{align*}
  \RTzerotwo_{C,n}(x_1,\dots,x_n) & := \bigwedge_{i \in [n]} \rho_{0,2}(x_i,x_{i+1 \bmod{n}}), \\
  \RTzerotwo_{K,n}(x_1,\dots,x_n) & := \bigwedge_{\substack{i,j \in [n] \\ i \neq j}} \rho_{0,2}(x_i,x_j). \\
  \intertext{Furthermore, let}
  \RTzerotwo_n & := \RTzerotwo_{C,n} \times \RTzerotwo_{K,n}.
\end{align*}
The names $C$ and $K$ in the indices of the relations are in correspondance 
with the circular graph $C_n$ and the complete graph $K_n$ on $n$ vertices.
The relations $\RTzerotwo_{C,n}$ have the same definition as $R^k_{\uparrow}$ in 
\cite{CouceiroHaddadSchoelzelWaldhauser:2013:ISMVL}. The idea behind replacing
$R^k_{\downarrow}$ with $\RTzerotwo_{K,n}$ stemmed from the fact that with graphs the following
holds:
\begin{itemize}
  \item Let $n' > n \geq 3$ be two odd numbers. Then there is no graph homomorphism from $C_n$ into $C_{n'}$.
  \item For $n' > n \geq 3$ there is no graph homomorphism from $K_{n'}$ into $K_n$.
  \item For $n',n \geq 3$ there is no graph homomorphism from $K_n$ into $C_{n'}$.
\end{itemize}
The relation $\RTzerotwo_n$ represents in this model the disjoint union $C_n \uplus K_n$ of $C_n$ and $K_n$.
Let $G \to H$ denote the fact, that there is some graph homomorphism from $G$ to $H$. 
We consider the possible homomorphisms from $K_{n'} \uplus C_{n'}$ to $K_n \uplus C_n$. Then we see
\begin{itemize}
  \item
    for $n' > n \geq 3$ that $C_{n'} \uplus K_{n'} \not\to C_{n} \uplus K_{n}$, 
    since $K_{n'} \not\to K_{n}$ and $K_{n'} \not\to C_{n}$; and
  \item
    for $n > n' \geq 3$ that any homomorphism from $C_{n'} \uplus K_{n'}$ to $C_{n} \uplus K_{n}$
    actually maps into $K_{n}$, since $C_{n'} \not\to C_{n}$ and $K_{n'} \not\to C_{n}$.
    But for the construction of $\RTzerotwo_{n}$ this would mean that the first $n$ coordinates
    are not essential, a contradiction.
\end{itemize}

For $n = 5$ the relations $\RTzerotwo_{C,5}$ and $\RTzerotwo_{K,5}$ look like this:
\[ \arraycolsep=2.8pt
  \RTzerotwo_{C,5} = 
    \begin{pmatrix}
      0 & 0 & 0 & 0 & 0 & 1 & 0 & 0 & 1 & 0 & 1 \\
      0 & 0 & 0 & 0 & 1 & 0 & 0 & 1 & 0 & 1 & 0 \\
      0 & 0 & 0 & 1 & 0 & 0 & 1 & 0 & 1 & 0 & 0 \\
      0 & 0 & 1 & 0 & 0 & 0 & 0 & 1 & 0 & 0 & 1 \\
      0 & 1 & 0 & 0 & 0 & 0 & 1 & 0 & 0 & 1 & 0 
    \end{pmatrix} \quad 
  \RTzerotwo_{K,5} = 
    \begin{pmatrix}
      0 & 0 & 0 & 0 & 0 & 1 \\
      0 & 0 & 0 & 0 & 1 & 0 \\
      0 & 0 & 0 & 1 & 0 & 0 \\
      0 & 0 & 1 & 0 & 0 & 0 \\
      0 & 1 & 0 & 0 & 0 & 0 
    \end{pmatrix}
\]

\begin{lemma}
  Let $n \geq 2$. Then $\Pol \RTzerotwo_n = \Pol \RTzerotwo_{C,n} = \Pol \RTzerotwo_{K,n} = T_{0,2}$. 
\end{lemma}

\begin{proof}
  By construction we have that $T_{0,2} \subseteq \Pol \RTzerotwo_{C,n}$, $T_{0,2} \subseteq \Pol \RTzerotwo_{K,n}$, and
  $\Pol \RTzerotwo_n = \Pol \RTzerotwo_{C,n} \cap \Pol \RTzerotwo_{K,n}$.

  Since $\rho_{0,2} = \pr_{1,2} \RTzerotwo_{C,n} = \pr_{1,2} \RTzerotwo_{K,n}$ we obtain 
  $\Pol \RTzerotwo_{C,n} \subseteq T_{0,2}$, and $\Pol \RTzerotwo_{K,n} \subseteq T_{0,2}$. From this follows 
  $\Pol \RTzerotwo_{C,n} = T_{0,2} = \Pol \RTzerotwo_{K,n}$, and consequently $T_{0,2} = \Pol \RTzerotwo_n$.
\end{proof}

Let $\hat{\IN} := \{ n \in \IN \mid n \text{ odd}, n \geq 3 \}$.
Let $n \in \hat{\IN}$ and $M \subseteq \hat{\IN} \setminus \{n\}$ for the rest of this section. We want
to show that 
\begin{equation} \label{equation:T02Aim}
  \pPol \RTzerotwo_n \not\supseteq \bigcap_{m \in M} \pPol \RTzerotwo_m
\end{equation}
holds. We assume to the contrary, that \eqref{equation:T02Aim} is false.
This means that by Theorem~\ref{theorem:definability} we can write
\begin{equation} \label{equation:failingConstructionOfRn}
  \RTzerotwo_n := \{ \xx \in \2^{2n} \mid \xx_\ii \in \RTzerotwo_m \text{ for all } \ii \in \gamma_m \text{ and } m \in M \}
\end{equation}
for some auxiliary relations $\gamma_m$ for all $m \in M$.
Furthermore, we can assume that no condition is superfluous.


\begin{lemma} \label{lemma:seperateCK}
  Let $m \in M$ with $\gamma_m \neq \emptyset$, and $\ii \in \gamma_m$.

  Then $\ii_{[m]} \subseteq [n]$ or $\ii_{[m]} \subseteq [n+1,2n]$. \\
  Similarly, $\ii_{[m+1,2m]} \subseteq [n]$ or $\ii_{[m+1,2m]} \subseteq [n+1,2n]$.
\end{lemma}

\begin{proof}
  We only consider the first statement; the second one follows similarly.

  Assume the statement is not true. Then there is some $j \in [m]$ such that 
  $i_j \in [n]$ and $i_{j+1 \bmod{m}} \in [n+1,2n]$. By construction of $\RTzerotwo_n$ 
  (or, more specifically $\RTzerotwo_{C,n}$) this means, that $\rho_{0,2}(x_{i_j}, x_{i_{j+1 \bmod{m}}})$ holds,
  i.e., $x_{i_j}$ and $x_{i_{j+1 \bmod{m}}}$ can not both be $1$ at the same time. But by construction
  of $\RTzerotwo_n$ we have 
  \[
  (0,\dots,0,\atPos{i_j}{1},0,\dots,0,\atPos{i_{j+1 \bmod{m}}}{1},0,\dots,0) \in \RTzerotwo_n.
  \]
  This is a contradiction, and thus $\ii_{[m]} \subseteq [n]$ or $\ii_{[m]} \subseteq [n+1,2n]$.
\end{proof}

\begin{lemma} \label{lemma:R02mlessn}
  Let $m \in M$ and $\gamma_m \neq \emptyset$.

  Then $m < n$.
\end{lemma}

\begin{proof}
  Assume to the contrary, that $m \in M$, $\gamma_m \neq \emptyset$, and $m > 0$. 
  
  Then there is some $\ii \in \gamma_m$.
  By Lemma~\ref{lemma:seperateCK} we have 
  $\ii_{[m+1,2m]} \subseteq [n]$ or $\ii_{[m+1,2m]} \subseteq [n+1,2n]$. Thus $i_j = i_{j'}$ 
  for some $j,j' \in [m+1,2m]$ with $j \neq j'$. By construction of $\RTzerotwo_n$ 
  (or, more specifically $\RTzerotwo_{K,n}$) this means, that $\rho_{0,2}(x_{i_j}, x_{i_j})$ holds,
  i.e., $x_{i_j} = 0$. But by construction
  of $\RTzerotwo_n$ we have 
  \[
  (0,\dots,0,\atPos{i_j}{1},0,\dots,0) \in \RTzerotwo_n.
  \]
  This is a contradiction, and thus $m < n$.
\end{proof}

Since $\RTzerotwo_n$ is not a trivial relation, there is at least one $m < n$ with non-empty $\gamma_m$. 
Thus we can assume that $n \geq 5$.

\begin{lemma} \label{lemma:R02intoKn}
  Let $m \in M$, $m < n$, and $\ii \in \gamma_m \neq \emptyset$.

  Then $\ii_{[m]} \subseteq [n+1,2n]$, and $\ii_{[m+1,2m]} \subseteq [n+1,2n]$.
\end{lemma}

\begin{proof}
  We only consider the first statement; the second one follows similarly.

  Assume to the contrary that $\ii_{[m]} \not\subseteq [n+1,2n]$ holds.
  By Lemma~\ref{lemma:seperateCK} we have $\ii_{[m]} \subseteq [n]$.
  
  If $|\ii_{[m]}| \leq 2$ then there is some $j \in [m]$ with $i_j = i_{j+1 \bmod{m}}$,
  implying $x_{i_j} = 0$. But by construction
  of $\RTzerotwo_n$ we have 
  \[
  (0,\dots,0,\atPos{i_j}{1},0,\dots,0) \in \RTzerotwo_n.
  \]
  Thus $|\ii_{[m]}| \geq 3$. Since $\RTzerotwo_{C,n}$ has a circular structure, and $m \leq n-2$, we have some
  $j,j' \in [m]$ with $j' = j+1 \bmod{m}$ and $|i_j - i_{j'} \bmod{n}| \geq 2$.
  But $i_j, i_{j'} \in [n]$ and by construction of $\RTzerotwo_n$ 
  (or, more specifically $\RTzerotwo_{C,n}$) this means, that $\rho_{0,2}(x_{i_j}, x_{i_{j'}})$ holds,
  i.e., $x_{i_j}$ and $x_{i_{j'}}$ can not both be $1$ at the same time. But by construction
  of $\RTzerotwo_n$ we have 
  \[
  (0,\dots,0,\atPos{i_j}{1},0,\dots,0,\atPos{i_{j'}}{1},0,\dots,0) \in \RTzerotwo_n.
  \]
  This is a contradiction, and thus $\ii_{[m]} \not\subseteq [n]$.
\end{proof}

This shows that in the right hand side of \eqref{equation:failingConstructionOfRn} 
the variables $x_1,\dots,x_n$ are inessential. 
But this contradicts the fact, that these variables are essential in $\RTzerotwo_n$. Thus follows:
\begin{theorem} \label{theorem:T02:nonconstructible}
  $\pPol \RTzerotwo_n \not\supseteq \bigcap_{m \in M} \pPol \RTzerotwo_m$.
\end{theorem}

\begin{corollary}
  Let $X,Y \subseteq \hat{\IN}$ be non-empty sets. Then
  \[
  \bigcap_{n \in X} \pPol \RTzerotwo_n = \bigcap_{m \in Y} \pPol \RTzerotwo_m \iff X = Y.
  \]
\end{corollary}

\subsection{The subclones of $T_{0,2}$}

Now we will look at the intervals $\intervalStr{D}$ for all subclones $D \subseteq T_{0,2}$. 
We use the fact that $T_{0,2} \subseteq T_0 = \Pol \{0\}$, and let $T := \cPol (\{0\},\emptyset)$.
In this way the conditions \ref{enum:strongdownwards:totalempty} and \ref{enum:strongdownwards:QiHungry}
of Lemma~\ref{lemma:strongdownwards} are fulfilled due to Corollary~\ref{corollary:emptysetIsVeryHungry}. 
The only condition left to show is 
\ref{enum:strongdownwards:different} for the set $I$ defined by
\[
  I := \left\{ \bigcap_{n \in X} \pPol \RTzerotwo_n \mid X \subseteq \hat{\IN}, X \neq \emptyset \right\}.
\]

\begin{lemma}
  Let $n \in \hat{\IN}$ and $M \subseteq \hat{\IN} \setminus \{ n \}$.

  Then $\pPol \RTzerotwo_n \not\supseteq T \cap \bigcap_{m \in M} \pPol \RTzerotwo_m$.
\end{lemma}

\begin{proof}
  We need to show that there is some 
  \[ 
    F \in \left( T \cap \bigcap_{m \in M} \pPol \RTzerotwo_m \right) \setminus \pPol \RTzerotwo_n.
  \]
  By Theorem~\ref{theorem:T02:nonconstructible} we have that there is some $l$-ary partial function
  \[ f \in \left( \bigcap_{m \in M} \pPol \RTzerotwo_m \right) \setminus \pPol \RTzerotwo_n. \]

  If $\0 \notin \dom f$, then $F := f \in T$ and thus we are done.

  We now assume that $\0 \in \dom f$. Since $f \notin \pPol \RTzerotwo_n$
  there is some matrix $M$ with columns $\xx_1,\dots,\xx_l \in \RTzerotwo_n$ and 
  rows $\yy_1,\dots,\yy_{2n}$ such that 
  \begin{itemize}
    \item
      $f(\xx_1,\dots,\xx_l) \notin \RTzerotwo_n$, and
    \item
      $\yy_1,\dots,\yy_{2n} \in \dom f$.
  \end{itemize}
  Let $f' \leq f$ be defined by $\dom f' := \{ \yy_1,\dots,\yy_{2n} \}$ and $f'(\yy_i) := f(\yy_i)$ for all
  $i \in [2n]$. 
  Thus we see that \[ f' \in \left( \bigcap_{m \in M} \pPol \RTzerotwo_m \right) \setminus \pPol \RTzerotwo_n. \] 

  If $\0 \notin \dom f'$, then $F := f' \in T$ and thus we are done.

  Thus there is some $j \in [2n]$, such that $\yy_j = \0$. 
  We define the $(l+1)$-ary partial function $g$ by
  \begin{align*}
    \dom g & := \{ (1,\yy_j) \} \cup \{ (0,\yy_i) \mid i \in [2n] \setminus \{j\} \}, \\
    g(1,\yy_j) & := f'(\yy_j), \\
    g(0,\yy_i) & := f'(\yy_i) \text{ for all } i \in [2n] \setminus \{j\}.
  \end{align*}
  As $g \leq \nabla f' \leq \nabla f$ we see that $g \in \bigcap_{m \in M} \pPol \RTzerotwo_m$.

  Because $\xx_0 := (0,\dots,0,\atPos{j}{1},0,\dots,0) \in \RTzerotwo_n$ we have
  \[
  g(\xx_0,\xx_1,\dots,\xx_l) = f(\xx_1,\dots,\xx_l) \notin \RTzerotwo_n,
  \]
  but all points are defined. Therefore $g \notin \pPol \RTzerotwo_n$ holds, and this implies
  \[ g \in \left( \bigcap_{m \in M} \pPol \RTzerotwo_m \right) \setminus \pPol \RTzerotwo_n. \]

  If $\0 \notin \dom g$, then $F := g \in T$ and thus we are done.
  Otherwise, repeating the steps from $f'$ to $g$ yields finally a desired $F$.
\end{proof}

\begin{corollary} \label{corollary:subT02:ITdifferent}
  Let $X,Y \in I$ with $X \neq Y$. Then $X \cap T \neq Y \cap T$.
\end{corollary}

\begin{theorem} \label{theorem:subT02}
  Let $D \subseteq T_{0,2}$ be a clone on $\OpBool$.
  
  Then $\intervalStr{D}$ has the cardinality of the continuum.
\end{theorem}

\begin{proof}
  By Corollary~\ref{corollary:subT02:ITdifferent} and the properties of $T$ mentioned at the beginning 
  of this subsection all conditions of Lemma~\ref{lemma:strongdownwards} hold, and therefore 
  $|\intervalStr{D}| \geq |I|$.
\end{proof}

\section{Continuum on $\Lambda$} \label{section:lambda}

From the results of the previous sections we see that the clones $\Lambda$, $\Lambda \cap T_1$, $V$, 
and $V \cap T_0$ are the only clones for which we need to determine the size of $\intervalStr{C}$.
We will show in this section that $\intervalStr{\Lambda}$ and $\intervalStr{\Lambda \cap T_1}$
have both the cardinality of the continuum. By symmetry of Post's lattice this implies the same statement
for $\intervalStr{V}$ and $\intervalStr{V \cap T_0}$.

Creignou, Kolaitis and Zanuttini have given in \cite{Creignou:2008:SIB:1435010.1435254} the set
of relations defining the smallest element in the interval $\intervalStr{C}$ for 
each Boolean clone $C$. They call these the plain basis. Since the least element in $\intervalStr{C}$
is $\Str{C}$ for each total clone $C$, we can conclude from \cite{Creignou:2008:SIB:1435010.1435254}
that 
\[
\Str{\Lambda} = \pPol \{ \lambda_k \mid k \geq 1 \}
\]
where $\lambda_k(y,x_1,\dots,x_k) \equiv (y \lor \lnot x_1 \lor \dots \lor \lnot x_k)$.
Equivalently, $\lambda_k = \2^{k+1} \setminus \{ (0,1,\dots,1) \}$.
The clone $\Lambda$ is denoted by $E$ in \cite{Creignou:2008:SIB:1435010.1435254},
and the plain basis can be found in the entry $IE$ of Table 2.

Any $n$-ary relation $\rho$ in the partial co-clone of $\Str{\Lambda}$ 
can be constructed from a selection of $\lambda_k$, 
i.e., there are (possibly empty) $k+1$-ary auxiliary relations $\gamma_k$ on $[n]$ for each $k \geq 1$ such that 
\begin{equation} \label{equation:rhodef}
  \rho(x_1,\dots,x_n) = \bigwedge_{k \geq 1} \bigwedge_{\ii \in \gamma_k} \lambda_k(x_{i_1},\dots,x_{i_{k+1}}).
\end{equation}
Since $\lambda_k$ is totally symmetric on the last $k$ coordinates, and 
$\lambda_k(y,x_1,\dots,x_k) = \lambda_{k+1}(y,x_1,x_1,\dots,x_k)$,
the tuples $\ii \in \gamma_k$ can be represented by pairs of the form $(i_1, \{ i_j \mid j \in [2,k+1] \})$.
This notation makes the symmetry of the relation more obvious, and exposes the special element more visibly.

For such pairs $(i, J)$ with $i \in [n]$ and $J \subseteq [n]$ we can define the $n$-ary relation
$\lambda^n_{(i,J)}$ by
\[
  \lambda^n_{(i,J)}(x_1,\dots,x_n) \equiv (x_i \lor \bigvee_{j \in J} \lnot x_j).
\]
We note that $\lambda^n_{(i,J)} = \2^n$ whenever $i \in J$, due to the tautology $x_i \lor \lnot x_i$ in the
definition of $\lambda^n_{(i,J)}$.

Let $\Gamma \subseteq \{ (i,J) \mid i \in [n], J \subseteq [n], i \notin J \}$. Then we define the relation
$\lambda^n_\Gamma \in \RelsBool[n]$ by
\[ \lambda^n_\Gamma := \bigwedge_{(i,J) \in \Gamma} \lambda^n_{(i,J)} (x_1,\dots,x_n). \]
Then equation \eqref{equation:rhodef} holds if and only if there is some suitable $\Gamma$ with
\[
  \rho(x_1,\dots,x_n) = \lambda^n_\Gamma.
\]

\begin{lemma}
  Let $i \in [n]$, $J \subseteq J' \subseteq [n]$. Then $\lambda^n_{(i,J)} \subseteq \lambda^n_{(i,J')}$.
\end{lemma}

\begin{proof}
  Follows from the definition.
\end{proof}


\subsection{Monsters}

\newcommand{\RLambda}{R^{\Lambda}}

In this subsection we will define relations $\RLambda_m$ for $m \geq 3$, which will be independent
and be used to show that $\intervalStr{\Lambda}$ has the cardinality of the continuum.
The relations $\RLambda_m$ will be called \emph{monsters}, as they ``kill'' this problem.

Let $m \geq 3$. We define $\Gamma_m \subseteq \{ (i,J) \mid i \in [m+1], J \subseteq [m+1], i \notin J \}$
and $\RLambda_m \in \RelsBool[m+1]$ by
\begin{align*}
  \Gamma_m & := \{ (1,[2,m+1]) \} \cup{} \\
           & \phantom{{}:={}} \{ (i,\{1,j_1,j_2\}) \mid i,j_1,j_2 \in [2,m+1], |\{i,j_1,j_2\}| = 3 \}, \\
  \RLambda_m & := \lambda^{m+1}_{\Gamma_m}.
  \intertext{A more visual represention of $\Gamma_4$ and $\RLambda_4$ is given in Table~\ref{table:Gamma4}.\newline
  Furthermore, we define the ternary relation $\RLambda_\Lambda$ by}
  \RLambda_\Lambda & := \{(0,0,0),(0,0,1),(0,1,0),(1,1,1) \}.
\end{align*}

\begin{table}
  \[
    \begin{array}{ll|ccccc}
      i & J & 1 & 2 & 3 & 4 & 5 \\ 
      \hline 
      1 & 2,3,4,5 & 0 & 1 & 1 & 1 & 1 \\
      2 & 1,3,4   & 1 & 0 & 1 & 1 &   \\
      2 & 1,3,5   & 1 & 0 & 1 &   & 1 \\
      2 & 1,4,5   & 1 & 0 &   & 1 & 1 \\
      3 & 1,2,4   & 1 & 1 & 0 & 1 &   \\
      3 & 1,2,5   & 1 & 1 & 0 &   & 1 \\
      3 & 1,4,5   & 1 &   & 0 & 1 & 1 \\
      4 & 1,2,3   & 1 & 1 & 1 & 0 &   \\
      4 & 1,2,5   & 1 & 1 &   & 0 & 1 \\
      4 & 1,3,5   & 1 &   & 1 & 0 & 1 \\
      5 & 1,2,3   & 1 & 1 & 1 &   & 0 \\
      5 & 1,2,4   & 1 & 1 &   & 1 & 0 \\
      5 & 1,2,5   & 1 &   & 1 & 1 & 0
    \end{array}
  \]
  \caption{Visual representation of $\Gamma_4$ and of the forbidden tuples in $\RLambda_4$. 
    For example, the condition $(2,\{1,4,5\})$ forbids the tuples $(1,0,x_3,1,1)$ for all
    $x_3 \in \2$. That means that $(1,0,0,1,1),(1,0,1,1,1) \notin \RLambda_4$.}
  \label{table:Gamma4}
\end{table}

As shown by Blochina in \cite{Blochina:1970:Posta} (see also Section~10.2~\cite{Lau:2006})
the relation $\RLambda_\Lambda$ characterizes the clone $\Lambda$, i.e.,
  \[ \Lambda = \Pol \RLambda_\Lambda. \]

Now we give some properties of the relations $\RLambda_m$.
\begin{lemma}
  Let $m \geq 3$. Then $\pPol \RLambda_m \subseteq \pPol \RLambda_\Lambda$.
\end{lemma}

\begin{proof}
  We have the following connections:
  \begin{align*}
    \lambda^3_{\{(1,\{2,3\}), (2,\{1,3\})\}}(x_1,x_2,x_3) & = \RLambda_m(x_1,x_2,x_3,\dots,x_3) \\
    \lambda^2_{(1,\{2\})}(x_1,x_2) & = \RLambda_m(x_1,x_2,\dots,x_2) \\
    \lambda_1(x_1,x_2) & = \lambda^2_{(1,\{2\})}(x_1,x_2) \\
    \RLambda_\Lambda(x_1,x_2,x_3) & = \lambda^3_{\{(1,\{2,3\}), (2,\{1,3\})\}}(x_1,x_2,x_3) \land \\
    & \phantom{{}={}} \lambda_1(x_2,x_1) \land \lambda_1(x_3,x_1)
  \end{align*}
  Thus $\RLambda_\Lambda$ is constructible from $\RLambda_m$ as
  \[
  \RLambda_\Lambda = \{ \xx \in \2^3 \mid \xx_{(1,2,3,\dots,3)}, \xx_{(2,1,\dots,1)}, \xx_{(3,1,\dots,1)} \in \RLambda_m \}.
  \]
  Thus $\pPol \RLambda_m \subseteq \pPol \RLambda_\Lambda$ by Theorem~\ref{theorem:definability}.
\end{proof}

\begin{corollary}
  Let $m \geq 2$. Then $\Pol \RLambda_m = \Lambda$, i.e., $\pPol \RLambda_m \in \intervalStrUp{\Lambda}$. 
\end{corollary}

\begin{proof}
  Since $\RLambda_m$ can be constructed from $\{ \lambda_k \mid k \geq 1 \}$ and 
  $\pPol \{ \lambda_k \mid k \geq 1 \} = \Str{\Lambda}$ we have $\Lambda \subseteq \pPol \RLambda_m$,
  and thus $\Lambda \subseteq \pPol \RLambda_m \cap \OpBool$.

  On the other hand we have $\pPol \RLambda_m \subseteq \pPol \RLambda_\Lambda$ and thus 
  $\pPol \RLambda_m \cap \OpBool \subseteq \pPol \RLambda_\Lambda \cap \OpBool = \Pol \RLambda_\Lambda = \Lambda$. 
\end{proof}

\begin{lemma} \label{lemma:RLambda:properties}
  Let $m \geq 3$. Then the following properties hold.
  \begin{enumerate}[label=(\roman*)]
    \item \label{enum:RLambda:properties:1inR}
      $ (1,\dots,1) \in \RLambda_m$. 
    \item \label{enum:RLambda:properties:11011notinR}
      $(1,\dots,1,\atPos{i}{0},1,\dots,1) \notin \RLambda_m$ for all $i \in [m+1]$.
    \item \label{enum:RLambda:properties:0xinR}
      $ \{ 0 \} \times (\2^m \setminus \{(1,\dots,1)\}) \subseteq \RLambda_m$.
    \item \label{enum:RLambda:properties:00100100inR}
      $ (0,\dots,0,\atPos{i}{1},0,\dots,0,\atPos{j}{1},0,\dots,0) \in \RLambda_m$ for all 
      $i,j \in [m+1]$ with $i < j$.
  \end{enumerate}
\end{lemma}

\begin{proof} \mbox{}
  \begin{itemize}
    \item[\ref{enum:RLambda:properties:1inR}] 
      Since $\1 \in \lambda_{(i,J)}$ for any $i$ and $J$, we have $\1 \in \RLambda_m$.

    \item[\ref{enum:RLambda:properties:11011notinR}]
      If $i = 1$, then $(0,1,\dots,1) \notin \lambda^{m+1}_{(1,[2,m+1])} \supseteq \RLambda_m$. Otherwise, 
      if $\xx = (1,\dots,1,0,1,\dots,1)$ then $\xx_{(i,1,j_1,j_2)} = (0,1,1,1) \notin 
      \lambda^4_{(i,\{1,j_1,j_2\})}$. Thus $\xx \notin \lambda^{m+1}_{(i,\{1,j_1,j_2\})} \supseteq \RLambda_m$.

    \item[\ref{enum:RLambda:properties:0xinR}]
      By the definition of $\lambda^{m+1}_{(i,J)}$ we see that $(0,x_2,\dots,x_{m+1}) \in \lambda^{m+1}_{(i,J)}$
      if $1 \in J$. Thus 
      \begin{align*}
        (\{ 0 \} \times \2^m) \cap \RLambda_m & = (\{ 0 \} \times \2^m) \cap \lambda^{m+1}_{(1,[2,m+1])} \\
        & = \{ 0 \} \times (\2^m \setminus \{(1,\dots,1)\}).
      \end{align*}

    \item[\ref{enum:RLambda:properties:00100100inR}]
      Let $\xx = (0,\dots,0,1,0,\dots,0,1,0,\dots,0)$. Since the set $J$ contains at least three elements 
      for every condition $\lambda^{m+1}_{(i,J)}$ in the construction of $\RLambda_m$, there
      is some $j \in J$ with $\xx_j = 0$. Thus $\xx \in \lambda^{m+1}_{(i,J)}$, 
      and consequently $\xx \in \RLambda_m$.\qedhere
  \end{itemize}
\end{proof}

\subsection{Monsters are good}

Similar to the case of $T_{0,2}$ we want to show that there are continuum many strong partial clones
with total part equal to $\Lambda$.

Let $\hat{\IN} := \{ n \in \IN \mid n \geq 3 \}$.
Let $n \in \hat{\IN}$ and $M \subseteq \hat{\IN} \setminus \{n\}$ for the rest of this section. We want
to show that 
\begin{equation} \label{equation:monstersAim}
  \pPol \RLambda_n \not\supseteq \bigcap_{m \in M} \pPol \RLambda_m
\end{equation}
holds. We assume to the contrary, that \eqref{equation:monstersAim} is false. 
This means that by Theorem~\ref{theorem:definability} we can write
\begin{equation} \label{equation:failingConstructionOfRnLambda}
  \RLambda_n := \{ \xx \in \2^{n+1} \mid \xx_\ii \in \RLambda_m \text{ for all } \ii \in \gamma_m \text{ and } m \in M \}
\end{equation}
with some auxiliary relations $\gamma_m$ for all $m \in M$.
Furthermore, we can assume that no condition is superfluous.

\begin{lemma} \label{lemma:gammamNoIdents}
  Let $m \geq 3$, $\ii \in \gamma_m$, and distinct $j,j' \in [m+1]$.
  Then $i_j \neq i_{j'}$.
\end{lemma}

\begin{proof}
  Assume to the contrary that there are distinct $j,j' \in [m+1]$ with
  $i_j = i_{j'}$.

  There are a few cases distinguished by the size of the set $[\ii]$.
  For each $x \in [m+1]$ let $t_x := \{ y \in [m+1] \mid i_y = i_x \}$.
  \begin{itemize}
    \item
      $|\ii| = 1$. Since $(0,\dots,0),(1,\dots,1) \in \RLambda_m$ the condition $\xx_\ii \in \RLambda_m$ is
      superfluous in contradiction to the assumption for \eqref{equation:failingConstructionOfRnLambda}.
    \item
      $|\ii| = 2$. We have three subcases.
      \begin{itemize}
        \item
          $|t_1| = 1$. Then set $\xx := (0,\dots,0,\atPos{i_2}{1},0,\dots,0) \in \RLambda_n$. But we have
          $\xx_\ii = (0,1,\dots,1) \notin \RLambda_m$, i.e., this case can not appear in the construction
          of $\RLambda_n$.
        \item
          $|t_1| = 2$. For each constraint $\lambda_{x,Y}$ in the construction of $\RLambda_m$ we have
          some $y \in Y \setminus t_1$. Thus $\{x,y\} \subseteq [m+1] \setminus t_1$, i.e.,
          these coordinates get identified. Therefore this constraint is superfluous.

          Since this holds for every such constraint the complete condition $\xx_{\ii} \in \RLambda_m$ is superfluous.
        \item
          $|t_1| \geq 3$. Let $\{1,y_2,y_3\} := Y \subseteq t_1$ with $1 \in Y$ and $|Y|=3$.
          Let $z := \min ([m+1] \setminus t_1)$, 
          and define $\xx := (0,\dots,0,\atPos{i_1}{1},0,\dots,0) \in \RLambda_n$.

          From $\RLambda_m \subseteq \lambda^{m+1}_{z,Y}$ and 
          $(\xx_{\ii})_{(z,1,y_2,y_3)} = (0,1,1,1) \notin \lambda^4_{1,\{2,3,4\}}$ follows that
          $\xx_{\ii} \notin \RLambda_m$.
          This contradicts $\xx \in \RLambda_n$ and therefore this case can not happen.
      \end{itemize}
    \item 
      $|\ii| \geq 3$. 

      Since there are distinct $j,j' \in [m+1]$ with $i_j \neq i_{j'}$
      there is some $x \in [m+1]$ with $|t_1 \cup t_x| \geq 3$.

      Let $t' := t_1 \cup t_x$. Since $|\ii| \geq 3$ we have $t' \neq [\ii]$, and thus the proof
      for $|\ii|=2$ and $|t_1| \geq 3$ works if we replace $t_1$ by $t'$.
  \end{itemize}
\end{proof}

\begin{corollary} \label{corollary:gammamEmptyForBigM}
  Let $m > n$. Then $\gamma_m = \emptyset$.
\end{corollary}

\begin{lemma} \label{lemma:gammaMNoOneForSmallM}
  Let $m < n$ and $\ii \in \gamma_m$. Then $1 \notin [\ii]$.
\end{lemma}

\begin{proof}
  Assume to the contrary, that $1 \in [\ii]$. By Lemma~\ref{lemma:gammamNoIdents} there are no identifications,
  i.e., $i_j \neq i_{j'}$ for all distinct $j,j' \in [m+1]$.

  There are two cases
  \begin{itemize}
    \item 
      $i_1 = 1$. We may assume w.l.o.g. that $i_x = x$ for all $x \in [m+1]$.

      We define $\xx := (0,\underbrace{1,\dots,1}_m,0,\dots,0)$. Then $\xx_{\ii} = (0,1,\dots,1) \notin \RLambda_m$
      but $\xx \in \RLambda_n$. Thus this contradicts \eqref{equation:failingConstructionOfRnLambda}.

    \item
      $i_j = 1$ for some $j \in [2,m+1]$. W.l.o.g. let $j = 2$. 
      
      Let $\{u_1,u_2,u_3\} := \{ i_1,i_3,i_4 \}$ with $u_1 < u_2 < u_3$ and define
      \[ \xx := (0,0,\dots,0,\atPos{u_1}{1},0,\dots,0,\atPos{u_2}{1},0,\dots,0,\atPos{u_3}{1},0,\dots,0). \]
      Then $\xx_{\ii} = (1,0,1,1,0,\dots,0) \notin \RLambda_m$ since $\RLambda_m \subseteq \lambda^{m+1}_{2,\{1,3,4\}}$.
      But $\xx \in \RLambda_n$. Thus this contradicts \eqref{equation:failingConstructionOfRnLambda}.
  \end{itemize}
\end{proof}

\begin{theorem} \label{theorem:Lambda:nonconstructible}
  Let $n \in \hat{\IN}$ and $M \subseteq \hat{\IN} \setminus \{n\}$.
  
  Then $\pPol \RLambda_n \not\supseteq \bigcap_{m \in M} \pPol \RLambda_m$.
\end{theorem}

\begin{proof}
  From Corollary~\ref{corollary:gammamEmptyForBigM} and Lemma~\ref{lemma:gammaMNoOneForSmallM} follows
  that $1 \notin [\ii]$ for all $\ii \in \gamma_m$ and $m \in M$.
  Thus in the right hand side of \eqref{equation:failingConstructionOfRnLambda} 
  the variable $x_1$ is inessential. 
  But this contradicts the fact, that this variables is essential in $\RLambda_n$.
  Therefore \eqref{equation:failingConstructionOfRnLambda} is not true, and by Theorem~\ref{theorem:definability}
  follows the statement of this theorem.
\end{proof}

\begin{corollary}
  Let $X,Y \subseteq \hat{\IN}$ non-empty sets. Then
  \[
  \bigcap_{n \in X} \pPol \RLambda_n = \bigcap_{m \in Y} \pPol \RLambda_m \iff X = Y.
  \]
\end{corollary}

From this follows that $I$ has continuum cardinality and with $I \subseteq \intervalStr{\Lambda}$
we obtain the following statement.

\begin{theorem} \label{theorem:Lambdacontinuum}
  The interval $\intervalStr{\Lambda}$ has the cardinality of the continuum.
\end{theorem}

\begin{theorem} \label{theorem:LambdaT1continuum}
  The interval $\intervalStr{\Lambda \cap T_1}$ has the cardinality of the continuum.
\end{theorem}

\begin{proof}
  We have $c_0 \in \Lambda \setminus (\Lambda \cap T_1)$. Thus
  Lemma \ref{lemma:diffconstants} is applicable, and by \ref{theorem:Lambdacontinuum}
  follows that $\intervalStr{\Lambda \cap T_1}$ has the cardinality of the continuum.
\end{proof}

%
%
%
%
%
%
%
%
%
%

\section{Conclusion}

Combining Theorems~\ref{theorem:subT02}, \ref{theorem:Lambdacontinuum}, \ref{theorem:LambdaT1continuum}, 
\ref{theorem:subL}, 
and \ref{theorem:finiteIntervals} we obtain the Dichotomy result for intervals of strong partial clones.

\begin{theorem}
  Let $C$ be a total Boolean clone.

  Then $\intervalStr{C}$ is either finite or has the cardinality of the continuum. Furthermore,
  $\intervalStr{C}$ is finite if and only if $M \cap T_0 \cap T_1 \subseteq C$ or $S \cap T_0 \cap T_1 \subseteq C$.
\end{theorem}

\subsection{Open questions}

Does the dichotomy between finite intervals and intervals of continuum cardinality also hold if we consider
the clones on some set $A$ with $|A| \geq 3$? Or, do there exists some $A$ and some
total clone $C$ in $\OA$ such that the interval $\intervalStr{C}$ is countably infinite?
Another question in this direction is concerning the two different intervals $\intervalD{C}$ and $\intervalStr{C}$
for some total clone $C$ in $\OA$. Clearly, $|\intervalStr{C}| \leq |\intervalD{C}|$ holds. In the Boolean
case for each total clone $C$ either both intervals are finite, or 
both intervals have the cardinality of the continuum. 
But is this also the case on every $A$ with $|A| \geq 3$?

For some subclones of $L$, and (in principle) all subclones of $\Lambda$ and $V$, respectively, we have shown
a strong relation between the intervals. Let $C \in \{ L, \Lambda, V \}$ and $D$ a total Boolean clone
with $D \subseteq C$ and $D \notin \{ C_{01}, \Omega_1 \}$. Then there is some partial class $T$, such that
\[ (X \cap T) \cup \Str{D} \subseteq (Y \cap T) \cup \Str{D} \iff X \subseteq Y \]
and
\[ (X \cap T) \cup \Str{D} \in \intervalStr{D} \]
hold for all $X,Y \in \intervalStr{C}$. This means that there is some order-preserving embedding of the interval 
$\intervalStr{C}$ into $\intervalStr{D}$. The author would be interested, if such an embedding is possible
for all pairs Boolean clones $C$ and $D$ with $D \subseteq C$? Since in this paper the structure of the lattice
was used explicitely, for example for the subclones of $L$, a more difficult question arises: 
If the embedding is possible, can this be proven in general without
directly using the description of all clones? What about this statement for some $A$ with $|A| \geq 3$?

The partial classes introduced in Section~\ref{section:partialClasses} are an equivalent of the classes 
considered by Harnau  in 
\cite{Harnau:19857:Relationenpaare:I,Harnau:19857:Relationenpaare:II,Harnau:19857:Relationenpaare:III}.
In there he presents the Galois connection and also describes the closure operator for the relation pairs.
The difference on the relational side between clones and strong partial clones is the omission of
the projection operator. Does this also work when switching from classes to strong partial classes?

\providecommand{\noopsort}[1]{} \providecommand{\foreignlanguage}[2]{#2}

\pagebreak[3]
\appendix
\section{Finite intervals of strong clones} \label{section:finiteIntervalsDrawings}

In the Figures~\ref{figure:MT} and \ref{figure:ST} we present the two finite intervals
$\intervalStrUp{M \cap T_0 \cap T_1}$ and $\intervalStrUp{S \cap T_0 \cap T_1}$, respectively.
These were given in \cite{HaddadSimons:2003}, but the drawings have been improved to show the structure in
a better way. The following 
short-hand notation is used for some of these partial Boolean clones.
All unlabeled points can be written as the intersection of some of these.
\begin{align*}
  P_a    & := \pPol \{a\} \text{ for } a \in \{0,1\} \\
  P_{01} & := \pPol \{(0,1)\} \\
  P_{\leq} & := \pPol \{(0,0),(0,1),(1,1)\} \\
  P_{a\leq} & := \pPol \{(a,0,0),(a,0,1),(a,1,1)\}  \text{ for } a \in \{0,1\} \\
  P_{01\leq} & := \pPol \{(0,1,0,0),(0,1,0,1),(0,1,1,1)\} \\
  P_{\lambda} & := \pPol \{(0,1),(1,0)\} \\
  P_{a\lambda} & := \pPol \{(a,0,1),(a,1,0)\}  \text{ for } a \in \{0,1\} \\
  P_{01\lambda} & := \pPol \{(0,1,0,1),(0,1,1,0)\}
\end{align*}

\begin{figure}[p]
  \intervalMT{0.9}
\caption{The interval $\intervalStrUp{M \cap T_0 \cap T_1}$}
\label{figure:MT}
\end{figure}

\begin{figure}[p]
  \intervalST{0.9}
\caption{The interval $\intervalStrUp{S \cap T_0 \cap T_1}$}
\label{figure:ST}
\end{figure}

\end{document}